\documentclass[graybox,envcountsect,envcountsame]{svmult}

\usepackage[latin1]{inputenc}
\usepackage[OT2,T1]{fontenc}
\usepackage[francais]{babel}
\usepackage{amsmath,amssymb}

\usepackage{mathptmx}
\usepackage{helvet}
\usepackage{courier}
\usepackage{makeidx}
\usepackage{graphicx}
\usepackage{multicol}
\usepackage[bottom]{footmisc}

\newcommand{\mylangle}{\xy(0,0),(1.17,1.485)**[|(1.025)]@{-},(0,0),(1.17,-1.485)**[|(1.025)]@{-}\endxy\mkern2mu}
\newcommand{\myrangle}{\mkern2mu\xy(0,0),(-1.17,1.485)**[|(1.025)]@{-},(0,0),(-1.17,-1.485)**[|(1.025)]@{-}\endxy}

\usepackage[ps,dvips,arrow,matrix,tips,line,curve]{xy}
\entrymodifiers={+!!<0pt,\the\fontdimen22\textfont2>}
\SelectTips{cm}{10}

\usepackage{tocvsec2}
\usepackage[pdfborder={0 0 0}]{hyperref}

\newcommand{\pr}{\mathrm{pr}_1}
\newcommand{\Gm}{\mathbf{G}_\mathrm{m}}
\newcommand{\kbar}{{\mkern1mu\overline{\mkern-1mu{}k\mkern-1mu}\mkern1mu}}
\newcommand{\Kbar}{{\mkern1mu\overline{\mkern-1mu{}K\mkern-1mu}\mkern1mu}}
\newcommand{\mmu}{{\boldsymbol{\mu}}}
\newcommand{\Qbarre}{{\mkern1mu\overline{\mkern-1mu{}\Q\mkern-1mu}\mkern1mu}}
\newcommand{\Xbarre}{{\mkern.4mu\overline{\mkern-.4mu{}X\mkern-1mu}\mkern1mu}}
\newcommand{\Kbarre}{{\mkern.4mu\overline{\mkern-.4mu{}K\mkern-1mu}\mkern1mu}}
\newcommand{\Vbarre}{{\mkern.4mu\overline{\mkern-.4mu{}V\mkern-1mu}\mkern1mu}}
\newcommand{\Spec}{\mathop\mathrm{Spec}\nolimits}
\newcommand{\Gal}{\mathop\mathrm{Gal}\nolimits}
\newcommand{\Pic}{\mathop\mathrm{Pic}\nolimits}
\newcommand{\Br}{\mathop\mathrm{Br}\nolimits}

\newcommand{\Ker}{\mathop\mathrm{Ker}\nolimits}

\newcommand{\A}{{\mathbf A}}
\newcommand{\Q}{{\mathbf Q}}
\newcommand{\Z}{{\mathbf Z}}
\renewcommand{\P}{{\mathbf P}}

\renewcommand{\geq}{\geqslant}
\renewcommand{\emptyset}{\varnothing}
\DeclareMathOperator{\inv}{inv}
\numberwithin{equation}{section}
\mathcode`A="7041 \mathcode`B="7042 \mathcode`C="7043 \mathcode`D="7044
\mathcode`E="7045 \mathcode`F="7046 \mathcode`G="7047 \mathcode`H="7048
\mathcode`I="7049 \mathcode`J="704A \mathcode`K="704B \mathcode`L="704C
\mathcode`M="704D \mathcode`N="704E \mathcode`O="704F \mathcode`P="7050
\mathcode`Q="7051 \mathcode`R="7052 \mathcode`S="7053 \mathcode`T="7054
\mathcode`U="7055 \mathcode`V="7056 \mathcode`W="7057 \mathcode`X="7058
\mathcode`Y="7059 \mathcode`Z="705A
\renewcommand{\Omega}{\varOmega}
\renewcommand{\Gamma}{\varGamma}
\smartqed
{\setbox0\hbox{$ $}}\fontdimen16\textfont2=\fontdimen17\textfont2

\begin{document}

\title*{Une remarque sur les courbes de Reichardt--Lind et de Schinzel}

\author{Olivier Wittenberg}
\institute{Date: 2 mai 2010.
\at Olivier Wittenberg \at D\'epartement de math\'ematiques et applications,
\'Ecole normale sup\'erieure,
45~rue d'Ulm,
75230~Paris Cedex~05,
France, \email{wittenberg@dma.ens.fr}}

\makeatletter
\providecommand*{\toclevel@author}{0}
\providecommand*{\toclevel@title}{0}
\makeatother
\setcounter{tocdepth}{-1}
\maketitle
\setcounter{tocdepth}{1}

\section{Introduction}
\renewcommand{\theequation}{\thesection.\arabic{equation}}

Les courbes de Reichardt--Lind~\cite{reichardt} \cite{lind} et de Schinzel~\cite{schinzel} sont deux exemples célèbres de courbes projectives et lisses~$X$ sur~$\Q$
possédant des points réels et des points $p$\nobreakdash-adiques pour tout premier~$p$, mais pas de point rationnel ni même de diviseur de degré~$1$.
La première est la courbe de genre~$1$ intersection dans~$\P^3$ des quadriques d'équations
$w^2=xz$
et $2y^2 = x^2 - 17 z^2$.   La seconde est de genre~$3$: il s'agit de la courbe quartique plane d'équation $x^4-17z^4=2(y^2+4z^2)^2$.
Dans cette note, à l'aide des résultats de~\cite{ew} concernant l'algébricité des classes de cycles de sections du groupe fondamental arithmétique sur les corps locaux,
nous établissons le théorème suivant:

\begin{theorem}
\label{th:intro}
Soit~$X$ la courbe de Reichardt--Lind ou la courbe de Schinzel.  La suite exacte fondamentale
\begin{align}
\label{eq:suitefond}\tag{$*$}
\xymatrix{
1 \ar[r] & \pi_1(X \otimes_\Q \Qbarre) \ar[r] & \pi_1(X) \ar[r] & \Gal(\Qbarre/\Q) \ar[r] & 1
}
\end{align}
n'est pas scindée.
\end{theorem}

Dans le cas de la courbe de Schinzel, le théorème~\ref{th:intro} confirme la
prédiction donnée par la conjecture des sections de Grothendieck~\cite{grothfaltings}.
Dans le cas de la courbe de Reichardt--Lind, il répond à une question posée par Stix~\cite{stixbm}.

La courbe quartique de Schinzel est ainsi le second exemple connu d'une courbe projective et lisse~$X$ sur~$\Q$, de genre~$\geq 2$, telle que la suite~(\ref{eq:suitefond})
ne soit pas scindée bien que $X(\A_\Q)\neq\emptyset$.  Un premier exemple
avait en effet été construit par Harari, Szamuely et Flynn~\cite{haszaflynn}.
Les arguments de~\cite{haszaflynn} requièrent la connaissance d'informations fines sur l'arithmétique de la jacobienne de~$X$
(entre autres: finitude du groupe de Tate--Shafarevich, rang).
Il n'en est pas ainsi de la preuve du théorème~\ref{th:intro}, dont le seul ingrédient arithmétique global est la loi de réciprocité pour le groupe de Brauer de~$\Q$.

Au paragraphe~\ref{par:critere} de cette note, nous donnons un critère général pour que le groupe fondamental d'une courbe projective et lisse sur un corps de nombres
n'admette pas de section, sous l'hypothèse qu'une
obstruction de Brauer--Manin s'oppose à l'existence d'un diviseur de degré~$1$ sur cette courbe.  Nous démontrons ensuite le théorème~\ref{th:intro},
à l'aide de ce critère,
au paragraphe~\ref{par:demthintro}.

\bigskip
\noindent\textbf{Remerciements.} Le contenu de cette note fut en partie exposé à Heidelberg lors de la conférence PIA~2010 organisée par Jakob Stix,
que je remercie pour son hospitalité.
Je remercie également Tamás Szamuely pour ses commentaires utiles sur une première version du texte.

\bigskip
\noindent\textbf{Notations.} Si~$M$ est un groupe abélien, on note $M_{\mathrm{tors}}$ le sous-groupe de torsion de~$M$.
Tous les groupes de cohomologie apparaissant ci-dessous sont des groupes de cohomologie étale.
Si~$X$ est une courbe irréductible, projective et lisse sur un corps~$k$, on note $\Pic^1(X)$ le sous-ensemble de $\Pic(X)$ constitué des classes de degré~$1$
et $\Br(X)=H^2(X,\Gm)$ le groupe de Brauer de~$X$.
Lorsque~$k$ est un corps de nombres, on désigne par~$\Omega$ l'ensemble des places de~$k$ et, pour $v \in \Omega$, par~$k_v$ le complété de~$k$ en~$v$.
Dans cette situation, Manin a défini un accouplement
\begin{align}
\label{eq:accbm}
\Biggl( \prod_{v \in \Omega} \Pic(X\otimes_k k_v) \!\!\Biggr) \times \Br(X) \to \Q/\Z\rlap{\text{,}}
\end{align}
somme d'accouplements locaux
\begin{align}
\label{eq:acclocal}
\Pic(X \otimes_k k_v) \times \Br(X \otimes_k k_v) \to \Q/\Z\rlap{\text{;}}
\end{align}
le noyau à gauche de~(\ref{eq:accbm}) contient l'image diagonale de $\Pic(X)$ (cf.~\cite[(8-2)]{saito}).
On dit qu'un sous-groupe $B \subset \Br(X)$ (resp.~une classe $A \in \Br(X)$) est \emph{responsable d'une obstruction de Brauer--Manin à l'existence d'un diviseur de degré~$1$ sur~$X$}
si aucun élément de $\prod_{v \in \Omega} \Pic^1(X \otimes_k k_v)$ n'est orthogonal à~$B$ (resp.~à~$A$) pour~(\ref{eq:accbm}).

\section{Un critère pour l'absence de sections sur un corps de nombres}
\label{par:critere}

Dans ce paragraphe, nous supposons donnés
un corps de nombres~$k$,
une courbe~$X$ projective, lisse et géométriquement irréductible sur~$k$,
un entier $N \geq 1$ et un sous-groupe $\Gamma \subset H^2(X,\mmu_N)$.
Considérons les deux hypothèses suivantes:
\begin{itemize}
\medskip
\item[](BM)
L'ensemble $\Pic^1(X \otimes_k k_v)$ est non vide pour tout $v \in \Omega$ et
l'image de~$\Gamma$ par la flèche naturelle $p:H^2(X,\mmu_N) \to \Br(X)$ est responsable d'une obstruction de Brauer--Manin à l'existence d'un diviseur de degré~$1$ sur~$X$.
\end{itemize}

\begin{itemize}
\medskip
\item[](T)
Pour toute place~$v$ de~$k$ divisant~$N$ et tout $\gamma \in \Gamma$, il existe
 $\gamma_{v,0} \in H^2(k_v,\mmu_N)$ et $\gamma_{v,1} \in \Pic(X \otimes_k k_v)_{\mathrm{tors}}$
tels que l'image $\gamma_v$ de~$\gamma$ dans $H^2(X \otimes_k k_v,\mmu_N)$ s'écrive $$\gamma_v=\gamma_{v,0} + c(\gamma_{v,1})\text{,}$$
où~$c$ désigne l'application classe de cycle $c:\Pic(X \otimes_k k_v) \to H^2(X \otimes_k k_v,\mmu_N)$.
\end{itemize}

\medskip
L'hypothèse~(BM) entraîne que $\Pic^1(X)=\emptyset$ et, par conséquent, que $X(k)=\emptyset$.
D'après Grothendieck~\cite{grothfaltings}, il devrait s'ensuivre, si le genre de~$X$ est $\geq 2$,
que la suite exacte fondamentale
\begin{align}
\label{eq:suitefondk}
\xymatrix{
1 \ar[r] & \pi_1(X \otimes_k \kbar) \ar[r] & \pi_1(X) \ar[r] & \Gal(\kbar/k) \ar[r] & 1
}
\end{align}
(cf.~\cite[Exp.~IX, 6.1]{sga1})
n'est pas scindée.
Nous montrons dans le théorème~\ref{th:critere} que tel est bien le cas si en outre l'hypothèse~(T) est satisfaite (et ce, même si~$X$ est de genre~$1$).

\begin{theorem}
\label{th:critere}
Soit~$X$ une courbe projective, lisse et géométriquement irréductible, sur un corps de nombres~$k$.
Soient $N \geq 1$ un entier et $\Gamma \subset H^2(X,\mmu_N)$ un sous-groupe.
Si~(BM) et~(T) sont vérifiées, la suite~(\ref{eq:suitefondk}) n'est pas scindée.
\end{theorem}

Ainsi, étant données une courbe~$X$ et une classe $A \in \Br(X)$ responsable d'une obstruction de Brauer--Manin à l'existence d'un diviseur
de degré~$1$ sur~$X$, pour que la suite~(\ref{eq:suitefondk}) ne soit pas scindée, il suffit qu'il existe un entier $N \geq 1$ et un relèvement
$\gamma \in H^2(X,\mmu_N)$ de~$A$ tels que pour toute place~$v$ de~$k$ divisant~$N$, l'image de~$\gamma$ dans $H^2(X \otimes_k k_v,\mmu_N)$ appartienne au sous-groupe
engendré par $H^2(k_v,\mmu_N)$ et par $c(\Pic(X \otimes_k k_v)_{\mathrm{tors}})$.

\spnewtheorem*{proofthcritere}{Preuve du théorème~\noexpand\ref{th:critere}}{\itshape}{\rmfamily}
\begin{proofthcritere}
Par l'absurde, supposons la suite~(\ref{eq:suitefondk}) scindée et fixons-en une section $s:\Gal(\kbar/k)\to \pi_1(X)$.
L'hypothèse~(BM) entraîne que la courbe~$X$ est de genre~$\geq 1$; c'est donc un $K(\pi,1)$ et l'on peut associer à~$s$ une classe
de cohomologie étale $\alpha \in H^2(X,\mmu_N)$ (cf.~\cite[2.6]{ew}).
Rappelons que~$\alpha$,
dite \emph{classe de cycle de~$s$},
est caractérisée par la propriété suivante: il existe un revêtement étale $f : Y \to X$ tel que~$s$ se
factorise par~$\pi_1(Y)$ et que, notant $\pr : X \times_k Y \to X$ la première projection, la classe dans
$H^2(X \times_k Y,\mmu_N)$
du graphe de~$f$ soit égale à
$\pr^\star\alpha$.

Pour $v \in \Omega$, notons $c:\Pic(X \otimes_k k_v) \to H^2(X \otimes_k k_v,\mmu_N)$ l'application classe de cycle
et~$\alpha_v$ l'image de~$\alpha$ dans $H^2(X \otimes_k k_v,\mmu_N)$.

\begin{proposition}
\label{propalg}
Pour tout $v \in \Omega$ ne divisant pas~$N$,
il existe $D_v \in \Pic^1(X \otimes_k k_v)$ tel que $\alpha_v=c(D_v)$.
\end{proposition}

\begin{proof}
Soit $v \in \Omega$ ne divisant pas~$N$.  Notons~$\kbar_v$ une clôture algébrique de~$k_v$ contenant~$\kbar$.  La flèche verticale de gauche du diagramme commutatif
\begin{align*}
\xymatrix@R=2em{
1 \ar[r] & \pi_1(X \otimes_k \kbar_v) \ar[d] \ar[r] & \pi_1(X \otimes_k k_v) \ar[d] \ar[r] & \Gal(\kbar_v/k_v) \ar[d] \ar[r] & 1 \\
1 \ar[r] & \pi_1(X \otimes_k \kbar) \ar[r] & \pi_1(X) \ar[r] & \Gal(\kbar/k) \ar[r] & 1
}
\end{align*}
étant un isomorphisme (cf.~\cite[Exp.~X, 1.8]{sga1}), la section~$s$ induit une section $s_v : \Gal(\kbar_v/k_v) \to \pi_1(X \otimes_k k_v)$ de la première ligne de ce diagramme.
À l'aide de la caractérisation de la classe de cycle d'une section rappelée ci-dessus,
on voit que la classe de cycle de~$s_v$ est égale à~$\alpha_v$.
Comme~$v$ ne divise pas~$N$, il résulte maintenant de \cite[Cor.~3.4, Rem.~3.7~(ii)]{ew}, si~$v$ est finie, ou de \cite[Rem.~3.7~(iv)]{ew}, si~$v$ est réelle, qu'il existe
$D_{v,0} \in \Pic(X \otimes_k k_v)$ tel que $\alpha_v = c(D_{v,0})$.
L'image de~$\alpha_v$ dans $H^2(X \otimes_k \kbar_v,\mmu_N)=\Z/N\Z$ est égale à~$1$
(cf.~\cite[2.6]{ew}).  Par conséquent $\deg(D_{v,0})=1+Nm$ pour un $m \in \Z$.  Soit $D_{v,1}$ un élément de l'ensemble $\Pic^1(X \otimes_k k_v)$, non vide par hypothèse.
Posons $D_v = D_{v,0} - Nm D_{v,1}$.
On a alors bien $\deg(D_v)=1$ et $c(D_v)=\alpha_v$.
\qed\end{proof}

Fixons, pour chaque $v \in \Omega$ divisant~$N$, un élément
arbitraire $D_v \in \Pic^1(X \otimes_k k_v)$,
et pour chaque $v \in \Omega$ ne divisant pas~$N$, un élément $D_v \in \Pic^1(X \otimes_k k_v)$ tel que $\alpha_v=c(D_v)$.
Nous allons maintenant démontrer que la famille $(D_v)_{v \in \Omega}$ est orthogonale,
pour l'accouplement~(\ref{eq:accbm}), à l'image de~$\Gamma$ par $p:H^2(X,\mmu_N) \to \Br(X)$.

Pour toute extension $K/k$, notons $$\mylangle -,- \myrangle : H^2(X \otimes_kK,\mmu_N) \times H^2(X \otimes_kK,\mmu_N) \to \Br(K)$$ la composée
du cup-produit
$H^2(X \otimes_kK,\mmu_N) \times H^2(X \otimes_kK,\mmu_N) \to H^4(X \otimes_k K, \mmu_N^{\otimes 2})$
et de la flèche $\delta:H^4(X \otimes_k K, \mmu_N^{\otimes 2}) \to H^2(K, H^2(X \otimes_k \Kbar,\mmu_N^{\otimes 2}))
=H^2(K,\mmu_N) \subset \Br(K)$ issue de la suite spectrale de Hochschild--Serre.
D'autre part, pour $v \in \Omega$, notons $\inv_v : \Br(k_v) \hookrightarrow \Q/\Z$ l'invariant de la théorie du corps de classes local
et notons encore $p : H^2(X \otimes_k k_v,\mmu_N) \to \Br(X \otimes_k k_v)$ la flèche naturelle.

\begin{lemma}
\label{lem:compatible}
Soit $v \in \Omega$.
Les accouplements $\mylangle -,- \myrangle$ et~(\ref{eq:acclocal})
s'inscrivent dans
un diagramme commutatif
$$
\xymatrix@R=1.5em@C=8ex{
\ar@<6ex>[d]^p H^2(X \otimes_k k_v, \mmu_N) \times H^2(X \otimes_k k_v,\mmu_N) \ar[r] & \Br(k_v) \ar[d]^{\inv_v} \\
\mkern-5mu \ar@<6ex>[u]^c \Pic(X \otimes_k k_v) \times \Br(X\otimes_k k_v) \ar[r] & \Q/\Z\rlap{\text{.}}
}
$$
\end{lemma}

\begin{proof}
Soient $x \in X \otimes_k k_v$ un point fermé et $i : \Spec(k_v(x)) \hookrightarrow X \otimes_k k_v$ l'injection canonique.
Pour tout $y \in H^2(X \otimes_k k_v,\mmu_N)$,
on a $\mylangle c(x), y \myrangle = \delta ( c(x) \smile y ) = \delta ( i_\star i^\star y)$, où la seconde égalité résulte de la formule de projection.  Compte tenu de la définition
de~(\ref{eq:acclocal}) (cf.~\cite[p.~399]{saito}), il suffit donc, pour conclure, de vérifier que l'application $\delta \circ i_\star : H^2(k_v(x),\mmu_N) \to H^2(k_v,\mmu_N)$
s'identifie au morphisme de corestriction de~$k_v(x)$ à~$k_v$.

Notons
 $\rho : X \otimes_k k_v \to \Spec(k_v)$ le morphisme structural
et
 $a:i_\star \mmu_N \to \mmu_N^{\otimes 2}[2]$ la flèche,
dans la catégorie
dérivée des faisceaux étales en groupes abéliens sur~$X$,
donnant naissance au morphisme de Gysin $i_\star : H^2(k_v(x),\mmu_N) \to H^4(X \otimes_k k_v,\mmu_N^{\otimes 2})$.
La composée de $R\rho_\star a$
et de la troncation $R\rho_\star \mmu_N^{\otimes 2}[2] \to \left(\tau_{\geq 2}R\rho_\star \mmu_N^{\otimes 2}\right)[2] = \mmu_N$
est une flèche entre complexes concentrés en degré~$0$. Elle provient donc d'un unique morphisme de modules galoisiens $b : (\rho \circ i)_\star \mmu_N \to \mmu_N$.
La flèche obtenue en appliquant à~$b$ le foncteur $H^0(\kbar_v,-)$ est la composée
 $$\mmu_N(x \otimes_{k_v} \kbar_v) \to H^2(X \otimes_k \kbar_v,\mmu_N^{\otimes 2}) \to \mmu_N$$
des applications classe de cycle et degré tordues par~$\mmu_N$.
Par conséquent~$b$ est l'application norme.
D'autre part, en appliquant à~$b$ le foncteur $H^2(k_v,-)$, on retrouve le morphisme $\delta \circ i_\star$; ainsi le lemme est prouvé.
\qed\end{proof}

Soit $\gamma \in \Gamma$.  Pour $v \in \Omega$, notons $\gamma_v$ l'image de~$\gamma$ dans $H^2(X \otimes_k k_v,\mmu_N)$.

\begin{proposition}
\label{prop:laprop}
Pour tout $v \in \Omega$, on a \kern.15ex$\mylangle \alpha_v, \gamma_v \myrangle = \mylangle c(D_v), \gamma_v \myrangle$.
\end{proposition}

\begin{proof}
Soit $v \in \Omega$.
Si~$v$ ne divise pas~$N$, on a même $\alpha_v=c(D_v)$.  Supposons donc que~$v$ divise~$N$ et reprenons les notations $\gamma_{v,0}$, $\gamma_{v,1}$ apparaissant dans l'hypothèse~(T),
de sorte que $\gamma_v = \gamma_{v,0} + c(\gamma_{v,1})$.

\begin{lemma}
\label{lem:annule}
Soient $x \in H^2(X \otimes_k k_v,\mmu_N)$ et $y \in H^2(k_v,\mmu_N)$.  Si~$x$ s'annule dans $H^2(X\otimes_k \kbar_v,\mmu_N)$, alors \kern.15ex$\mylangle x,y \myrangle=0$.
\end{lemma}

\begin{proof}
Notons $(F^iH^{2n})_{i \in \Z}$ la filtration décroissante de $H^{2n}(X \otimes_k k_v, \mmu_N^{\otimes n})$ induite par la suite spectrale de Hochschild--Serre.
Les hypothèses du lemme signifient que $x \in F^1H^2$ et $y \in F^2H^2$.
Or le cup-produit respecte cette filtration (cf.~\cite[Prop.~6.2]{jannsen});
d'où $x \smile y \in F^3H^4$, ce qui se traduit par l'égalité \kern.15ex$\mylangle x,y\myrangle=0$.
\qed\end{proof}

Le construction de la classe de cycle associée à une section du groupe fondamental étant fonctorielle par rapport aux coefficients,
la classe $\alpha_v$ se relève dans
$\varprojlim_{m \geq 1} H^2(X \otimes_k k_v,\mmu_m)$
(cf.~\cite[2.6]{ew}).
Par conséquent $p(\alpha_v) \in \Br(X \otimes_k k_v)$ est infiniment divisible. En particulier $p(\alpha_v)$ est orthogonal à $\Pic(X \otimes_k k_v)_{\mathrm{tors}}$
pour l'accouplement~(\ref{eq:acclocal}).
Or $\gamma_{v,1} \in \Pic(X \otimes_k k_v)_{\mathrm{tors}}$ (voir l'hypothèse~(T));
le lemme~\ref{lem:compatible} implique donc que \kern.15ex$\mylangle \alpha_v, c(\gamma_{v,1})\myrangle=0$.

D'autre part, il résulte du lemme~\ref{lem:compatible} que
\kern.15ex$\mylangle c(D_v), c(\gamma_{v,1})\myrangle=0$, puisque $p \circ c=0$.
Enfin, le lemme~\ref{lem:annule} entraîne que
\kern.15ex$\mylangle \alpha_v - c(D_v), \gamma_{v,0}\myrangle=0$.
Vu la décomposition
$$\mylangle \alpha_v - c(D_v), \gamma_{v,0} + c(\gamma_{v,1}) \myrangle=
\mylangle \alpha_v - c(D_v), \gamma_{v,0}\myrangle + \mylangle \alpha_v, c(\gamma_{v,1}) \myrangle - \mylangle c(D_v), c(\gamma_{v,1})\myrangle\rlap{\text{,}}$$
la proposition~\ref{prop:laprop} est maintenant établie.
\qed\end{proof}

Nous sommes en position d'achever la preuve du théorème~\ref{th:critere}.
D'après la loi de réciprocité globale, la somme des invariants de
$\mylangle \alpha, \gamma \myrangle \in \Br(k)$ est nulle.  Par conséquent
 $\sum_{v \in \Omega}\inv_v \mylangle \alpha_v, \gamma_v \myrangle=0$.
Il s'ensuit, grâce à la proposition~\ref{prop:laprop},
que $\sum_{v \in \Omega} \inv_v \mylangle c(D_v), \gamma_v \myrangle=0$,
ce qui signifie,
compte tenu du lemme~\ref{lem:compatible}, que la famille
$(D_v)_{v \in \Omega}$ est orthogonale à $p(\gamma)$ pour l'accouplement~(\ref{eq:accbm}).
L'élément $\gamma \in \Gamma$ étant quelconque, l'hypothèse~(BM) est ainsi contredite.
\qed\end{proofthcritere}

\section{Les courbes de \texorpdfstring{Reichardt--Lind}{Reichardt-Lind} et de Schinzel}
\label{par:demthintro}

Il existe un morphisme évident de la courbe de Schinzel vers la courbe de Reichardt--Lind, à savoir
$[x:y:z] \mapsto [xz : x^2: y^2+4z^2: z^2]$.
Or la suite exacte~(\ref{eq:suitefond}) dépend fonctoriellement de~$X$. Pour établir le théorème~\ref{th:intro},
il suffit donc de traiter le cas de la courbe de Reichardt--Lind.

Soit~$X$ la courbe de Reichardt--Lind, définie comme la compactification lisse de la courbe affine
d'équation $2y^2=w^4-17$ sur~$\Q$.  Pour montrer que la suite~(\ref{eq:suitefond}) n'est pas scindée, nous allons
appliquer le théorème~\ref{th:critere} avec $N=2$ et $\Gamma=\{0,\gamma\}$ où $\gamma \in H^2(X,\mkern1mu\Z/2\Z)$ est une classe à préciser.
Le reste du présent paragraphe est consacré à la construction de~$\gamma$ et à la vérification des hypothèses du théorème~\ref{th:critere}.

Le diviseur de la fonction rationnelle~$y$ sur~$X$ s'écrit $P-2Q$, où $P,Q \in X$ sont des points
fermés de degrés respectifs~$4$ et~$2$ sur~$\Q$.  Posons $V=X \setminus \{P\}$.
Le diviseur de~$y$ sur~$V$ étant un double, le revêtement de~$V$ obtenu en extrayant une racine
carrée de~$y$ est étale.  Notons $[y]$ sa classe dans $H^1(V,\mkern1mu\Z/2\Z)$.  Notons d'autre part $[17]$ l'image de~$17$ par l'application
naturelle
$\Q^\star/\Q^{\star 2}=H^1(\Q,\mkern1mu\Z/2\Z) \to H^1(V,\mkern1mu\Z/2\Z)$.

\begin{lemma}
\label{lem:inject}
La flèche de restriction $H^2(X,\mkern1mu\Z/2\Z) \to H^2(V,\mkern1mu\Z/2\Z)$ est injective.
\end{lemma}

\begin{proof}
En effet, dans la suite exacte de localisation
$$
\xymatrix{
H^1(V,\mkern1mu\Z/2\Z) \ar[r] & H^2_P(X,\mkern1mu\Z/2\Z) \ar[r] & H^2(X,\mkern1mu\Z/2\Z) \ar[r] & H^2(V,\mkern1mu\Z/2\Z)\rlap{\text{,}}
}
$$
la première flèche est surjective puisque
$H^2_P(X,\mkern1mu\Z/2\Z) = H^0(P,\mkern1mu\Z/2\Z) = \Z/2\Z$ et que $[y] \in H^1(V,\mkern1mu\Z/2\Z)$ s'envoie sur $1 \in \Z/2\Z$.
\qed\end{proof}

Le résidu de $[y] \smile [17] \in H^2(V,\mkern1mu\Z/2\Z)$ en~$P$ est nul puisque~$17$ est un carré dans~$\Q(P)$.  Par conséquent $[y] \smile [17]$ est la restriction
d'un élément de $H^2(X,\mkern1mu\Z/2\Z)$.  D'après le lemme~\ref{lem:inject},
celui-ci est uniquement déterminé. Nous le noterons~$\gamma$.

L'image de $\gamma$ dans $\Br(X)$ est la classe de l'algèbre de quaternions $(y,17)$.  Il est bien connu que
cette classe est responsable d'une obstruction de Brauer--Manin à l'existence d'un diviseur de degré~$1$ sur~$X$  et que
$X(\Q_v)\neq\emptyset$ pour toute place~$v$ de~$\Q$ (cf.~\cite[\textsection5]{stixbm}, où l'obstruction de Brauer--Manin à l'existence d'un point
rationnel est discutée; noter que  $\Pic^1(X \otimes_\Q K)=X(K)$ pour toute extension $K/\Q$ puisque~$X$ est une courbe de genre~$1$; ainsi l'obstruction de Brauer--Manin
à l'existence d'un diviseur de degré~$1$ sur~$X$ est équivalente à l'obstruction de Brauer--Manin à l'existence d'un point rationnel).
L'hypothèse~(BM) du paragraphe~\ref{par:critere} est donc satisfaite.  Il reste à vérifier l'hypothèse~(T).

Afin de simplifier les notations, posons $\Xbarre = X \otimes_\Q \Qbarre$, $\Vbarre = V \otimes_\Q \Qbarre$ et enfin
$$H^2(X \otimes_\Q K,\mkern1mu\Z/2\Z)^0=\Ker\left(H^2(X \otimes_\Q K,\mkern1mu\Z/2\Z) \to H^2(X \otimes_\Q \Kbarre,\mkern1mu\Z/2\Z)\right)$$
pour toute extension $K/\Q$, où~$\Kbarre$ désigne une clôture algébrique de~$K$.

\begin{proposition}
\label{prop:degzero}
L'image de $\gamma$ dans $H^2(\Xbarre,\mkern1mu\Z/2\Z)$ est nulle.
\end{proposition}

\begin{proof}
Le diviseur de la fonction rationnelle $w^2 - \sqrt{17}$ sur $\Xbarre$ est un double; celle-ci définit donc une classe $[w^2-\sqrt{17}] \in H^1(\Xbarre,\mkern1mu\Z/2\Z)$.
Cette classe est invariante sous l'action de $\Gal(\Qbarre/\Q)$ puisque $(w^2-\sqrt{17})(w^2+\sqrt{17})=w^4-17=2y^2$ est un carré dans $\Qbarre(X)$.

\begin{lemma}
\label{lem:cocycle}
Les images de $[w^2-\sqrt{17}] \times [-2]$ et de $[y] \times [17]$ par le cup-produit
\begin{align}
\label{eq:cupprod}
H^0(\Q,H^1(\Vbarre,\mkern1mu\Z/2\Z)) \times H^1(\Q,\mkern1mu\Z/2\Z) \to H^1(\Q,H^1(\Vbarre,\mkern1mu\Z/2\Z))
\end{align}
coïncident.
\end{lemma}

\begin{proof}
Il suffit de montrer que le cocycle $a:\Gal(\Qbarre/\Q) \to H^1(\Vbarre,\mkern1mu\Z/2\Z)$
défini par
$a(\sigma) = \chi_{-2}(\sigma)[w^2-\sqrt{17}] - \chi_{17}(\sigma)[y]$
est un cobord, où $\chi_q : \Gal(\Qbarre/\Q) \to \Z/2\Z$ désigne le caractère quadratique associé à $q \in \Q^{\star}$.
Soit $f = w^2 - \sqrt{17} - \sqrt{-2}y$.
Le~diviseur de~$f$ sur~$\Vbarre$ est un double. D'où une classe $[f] \in H^1(\Vbarre,\mkern1mu\Z/2\Z)$.
On vérifie aisément que $a(\sigma)=\sigma[f]-[f]$ pour tout~$\sigma$
tel que l'un au moins  de $\chi_{-2}(\sigma)$ et de $\chi_{17}(\sigma)$ soit nul.
Or~$a$ est un cocycle; par conséquent
 $a(\sigma)=\sigma[f]-[f]$ pour tout $\sigma \in \Gal(\Qbarre/\Q)$ et~$a$ est donc un cobord.
\qed\end{proof}

Comme $H^2(\Vbarre,\mkern1mu\Z/2\Z)=0$,
la suite spectrale de Hochschild--Serre et les flèches de restriction fournissent un diagramme commutatif
\begin{align*}
\xymatrix@C=4ex{
H^2(X,\mkern1mu\Z/2\Z)^0 \ar[d] \ar[r]^(.42){\delta_X} & H^1(\Q,H^1(\Xbarre,\mkern1mu\Z/2\Z)) \ar[d] \ar[r] & H^3(\Q,H^0(\Xbarre,\mkern1mu\Z/2\Z)) \ar[d]^(.45)\wr \\
H^2(V,\mkern1mu\Z/2\Z) \ar[r]^(.42){\delta_V} & H^1(\Q,H^1(\Vbarre,\mkern1mu\Z/2\Z)) \ar[r]^d & H^3(\Q,H^0(\Vbarre,\mkern1mu\Z/2\Z))
}
\end{align*}
dont les lignes sont exactes et dont la flèche verticale de droite est un isomorphisme.
Cette suite spectrale étant compatible au cup-produit
(cf.~\cite[Prop.~6.2]{jannsen}), l'image de $[y] \times [17]$ par~(\ref{eq:cupprod}) est égale, à un signe près, à $\delta_V([y] \smile [17])$;
en particulier appartient-elle à $\Ker(d)$.
Il s'ensuit, grâce au lemme~\ref{lem:cocycle} et à une chasse au diagramme, que l'image de $[w^2-\sqrt{17}] \times [-2]$ par le cup-produit
$$
H^0(\Q,H^1(\Xbarre,\mkern1mu\Z/2\Z)) \times H^1(\Q,\mkern1mu\Z/2\Z) \to H^1(\Q,H^1(\Xbarre,\mkern1mu\Z/2\Z))
$$
s'écrit $\delta_X(\gamma\mkern2mu')$ pour un $\gamma\mkern2mu' \in H^2(X,\mkern1mu\Z/2\Z)^0$.
La flèche naturelle $\Ker(\delta_X) \to \Ker(\delta_V)$ est
surjective puisque tout élément de $\Ker(\delta_V)$ provient de $H^2(\Q,\Z/2\Z)$.
 Quitte à modifier~$\gamma\mkern2mu'$, on peut donc supposer que la restriction de $\gamma\mkern2mu'$ à~$V$
coïncide avec $[y] \smile [17] \in H^2(V,\mkern1mu\Z/2\Z)$.
Le lemme~\ref{lem:inject} entraîne maintenant que~$\gamma=\gamma\mkern2mu'$. D'où finalement $\gamma \in H^2(X,\mkern1mu\Z/2\Z)^0$.
\qed\end{proof}

\begin{proposition}
\label{prop:noyaurest}
Pour toute extension $K/\Q$, le noyau de la flèche de restriction
\begin{align}
\label{eq:rest}
H^2(X \otimes_\Q K,\mkern1mu\Z/2\Z)^0 \to H^2(V \otimes_\Q K,\mkern1mu\Z/2\Z)
\end{align}
est contenu dans
$c(\Pic(X \otimes_\Q K)_{\mathrm{tors}})$,
où~$c$ désigne l'application classe de cycle.
\end{proposition}

\begin{proof}
Le lieu de ramification du morphisme $X \to \P^1_\Q$, $(y,w)\mapsto w$ est~$P$, qui est de degré~$4$ sur~$\Q$.
Par conséquent, le choix d'un point géométrique de~$X$ au-dessus de~$P$ munit $X \otimes_\Q\Kbarre$ d'une structure de courbe elliptique
dont le sous-groupe de $2$\nobreakdash-torsion est $P \otimes_\Q \Kbarre$.
Les diviseurs de degré~$0$ sur $X \otimes_\Q K$ supportés
sur $P \otimes_\Q K$ sont donc tous de torsion dans $\Pic(X \otimes_\Q K)$.
Or leurs classes de cycles dans $H^2(X \otimes_\Q K,\mkern1mu\Z/2\Z)^0$ engendrent le noyau de~(\ref{eq:rest}), en vertu de la suite exacte de localisation.
\qed\end{proof}

L'image de $\gamma$ dans $H^2(V \otimes_\Q \Q_2,\mkern1mu\Z/2\Z)$ est nulle
puisque~$17$ est un carré dans~$\Q_2$.
Les propositions~\ref{prop:degzero} et~\ref{prop:noyaurest}
entraînent donc que
l'hypothèse~(T) est satisfaite. Ainsi le théorème~\ref{th:critere} permet-il de conclure la démonstration du théorème~\ref{th:intro}.

\begin{remark}
Soit~$X$ la courbe de Schinzel.  Nous avons montré que la suite~(\ref{eq:suitefond}) n'est pas scindée.
À l'instar de l'exemple de~\cite{haszaflynn}, la courbe de Schinzel vérifie une propriété plus forte:
même la suite exacte fondamentale abélianisée
\begin{align*}
\xymatrix{
1 \ar[r] & \pi_1(\Xbarre)^{\mathrm{ab}} \ar[r] & \pi_1(X)^{[\mathrm{ab}]} \ar[r] & \Gal(\Qbarre/\Q) \ar[r] & 1
}
\end{align*}
(obtenue en poussant~(\ref{eq:suitefond}) le long du morphisme d'abélianisation $\pi_1(\Xbarre) \to \pi_1(\Xbarre)^{\mathrm{ab}}$)
n'est pas scindée.  En effet, la suite exacte fondamentale abélianisée est tout aussi fonctorielle que la suite exacte fondamentale,
et les deux coïncident dans le cas de la courbe de Reichardt--Lind puisque celle-ci est de genre~$1$.
\end{remark}

\def\refname{Bibliographie}
\bibliographystyle{halpha}
\bibliography{sectionq}

\begin{thebibliography}{EW09}

\bibitem[EW09]{ew}
H{\'e}l{\`e}ne Esnault and Olivier Wittenberg.
\newblock Remarks on cycle classes of sections of the arithmetic fundamental
  group.
\newblock {\em Mosc. Math. J.}, 9(3):451--467, 2009.

\bibitem[Gro71]{sga1}
Alexander Grothendieck.
\newblock {\em Rev\^etements \'etales et groupe fondamental}.
\newblock Springer-Verlag, Berlin, 1971.
\newblock S{\'e}minaire de G{\'e}om{\'e}trie Alg{\'e}brique du Bois--Marie
  1960--1961 (SGA~1), Lecture Notes in Mathematics, Vol. 224.

\bibitem[Gro83]{grothfaltings}
Alexander Grothendieck.
\newblock Lettre \`a {F}altings du 27 juin 1983 (en allemand), parue dans:
  Geometric Galois actions, Vol.~1, London Math. Soc. Lecture Note Ser.,
  vol.~242, Cambridge Univ. Press, Cambridge, 1997.

\bibitem[HS09]{haszaflynn}
David Harari and Tam{\'a}s Szamuely.
\newblock Galois sections for abelianized fundamental groups.
\newblock {\em Math. Ann.}, 344(4):779--800, 2009.
\newblock Avec un appendice par E. V. Flynn.

\bibitem[Jan88]{jannsen}
Uwe Jannsen.
\newblock Continuous \'etale cohomology.
\newblock {\em Math. Ann.}, 280(2):207--245, 1988.

\bibitem[Lin40]{lind}
Carl-Erik Lind.
\newblock Untersuchungen \"uber die rationalen {P}unkte der ebenen kubischen
  {K}urven vom {G}eschlecht {E}ins.
\newblock Th\`ese de doctorat, Uppsala, 1940.

\bibitem[Rei42]{reichardt}
Hans Reichardt.
\newblock Einige im {K}leinen \"uberall l\"osbare, im {G}rossen unl\"osbare
  diophantische {G}leichungen.
\newblock {\em J. reine angew. Math.}, 184:12--18, 1942.

\bibitem[Sai89]{saito}
Shuji Saito.
\newblock Some observations on motivic cohomology of arithmetic schemes.
\newblock {\em Invent. math.}, 98(2):371--404, 1989.

\bibitem[Sch84]{schinzel}
Andrej Schinzel.
\newblock Hasse's principle for systems of ternary quadratic forms and for one
  biquadratic form.
\newblock {\em Studia Math.}, 77(2):103--109, 1984.

\bibitem[Sti09]{stixbm}
Jakob Stix.
\newblock The {B}rauer--{M}anin obstruction for sections of the fundamental
  group.
\newblock Pr\'epublication, 2009.

\end{thebibliography}
\end{document}